\documentclass[letter, 12pt]{amsart}

\usepackage[utf8]{inputenc}
\usepackage[T1]{fontenc}
\usepackage{graphicx}
\usepackage{amsfonts}
\usepackage{latexsym}
\usepackage{amsmath}
\usepackage{amssymb}
\usepackage{amsrefs}
\usepackage[all]{xy}

\usepackage[usenames]{color}
\usepackage{ifthen}

\newtheorem{thm}{Theorem}[section]

\newtheorem{lem}[thm]{Lemma}
\newtheorem{prop}[thm]{Proposition}
\theoremstyle{definition}
\newtheorem{defn}[thm]{Definition}
\theoremstyle{remark}

\numberwithin{equation}{section}

\newcommand{\set}[2]{ \left\{ #1 | #2 \right\} }

\newcommand{\TO}{\longrightarrow}
\newcommand{\To}{\longmapsto}

\newcommand{\Img}{\mathop{\mathrm{Im}}}

\newcommand{\Krn}{\mathop{\mathrm{Ker}}}

\newcommand{\K}{\mathbb K}
\newcommand{\vphi}{\varphi}
\newcommand{\tphi}{\tilde{\varphi}}
\newcommand{\bphi}{\bar{\varphi}}
\newcommand{\A}{A \otimes A}

\newcommand{\huang}[3]{\ifnum #1 =1 \textcolor{blue}{#3}\else {#2}\fi }
\newcommand{\brice}[3]{\ifnum #1 =1 \textcolor{green}{#3}\else {#2}\fi }

\begin{document}


\title[Direct sums of zero product determined algebras]{Direct sums of zero product determined algebras} 
\author{Daniel Brice and Huajun Huang} 
\address{Department of Mathematics and Statistics, Auburn University, Alabama, United States} 
\email{dpb0006@auburn.edu, \ huanghu@auburn.edu} 

\thanks{} 

\date{10 October 2011} 
\dedicatory{} 
\commby{} 

\begin{abstract} 

We reformulate the definition of a zero product determined algebra in terms of tensor products and obtain necessary and sufficient conditions for an algebra to be zero product determined.
These  conditions allow us to prove that the direct sum $\bigoplus_{i \in I} A_i$ of algebras for any index set $I$ is zero product determined if and only if each of the component algebras $A_i$ is zero product determined.
As an application, every parabolic subalgebra of a finite-dimensional reductive Lie algebra, over an algebraically-closed field of characteristic zero, is zero product determined. In particular, every such reductive Lie algebra is zero product determined.

\end{abstract}


\maketitle

Let $\K$ be a commutative ring. Given a $\K$-algebra $A$ and a $\K$-bilinear map $\vphi: A \times A \TO B$, we may ask whether or not $\vphi$ may be   written as the composition of multiplication in $A$ with a $\K$-linear transformation $\tphi$, that is, whether or not
$$\vphi(a_1,a_2) = \tphi(a_1a_2)$$
for all $a_1,a_2 \in A$ for some $\tphi : A^2 \TO B$.  (Here and throughout, $A^2$ denotes the $\K$-linear span of the products of members of $A$).

In order to study the above problem, Bre\v{s}ar, Gra\v{s}i\v{c}, and Ortega introduced the notion of a zero product determined algebra in \cite{Bresar09}.
A $\K$-algebra $A$ (not necessarily associative) is called \emph{zero product determined} if each $\K$-bilinear map $\vphi: A \times A \TO B$ satisfying
$$\vphi(a_1,a_2) = 0 \ \text{ whenever } \ a_1a_2 = 0$$
can be written as $\vphi(a_1,a_2) = \tphi(a_1a_2)$ for some $\tphi: A^2 \TO B$.
Their definition was motivated by applications to the study of zero product preserving linear maps defined on Banach algebras and on matrix algebras under the standard matrix product, the Lie product, or the Jordan product \cite{Alaminos09, Bresar09, Chebotar03, Chebotar06,Wang11b}.
Let $A, B$ be $\K$-algebras, and let $f: A \TO B$ be $\K$-linear.
$f$ is said to be zero product preserving if $f(a_1)f(a_2)=0$ whenever $a_1a_2=0$.
One would like to find conditions on $f$ or on $A$ that imply that $f$ is a homomorphism of $\K$-algebras, or is at least close to an algebra homomorphism in the following sense.
We define a mapping $\vphi(a_1,a_2)=f(a_1)f(a_2)$. Then $\vphi$ is $\K$-bilinear and satisfies $\vphi(a_1,a_2)=0$ whenever $a_1a_2=0$.
If $A$ is zero product determined, then there is a unique $\K$-linear $\tphi: A^2 \TO B$ satisfying
$$\tphi(a_1a_2)=\vphi(a_1,a_2)=f(a_1)f(a_2)$$
for all $a_1,a_2 \in A$.
If we further assume that $1_A \in A$ (so in particular $A^2=A$), then
$$ \tphi(a) = \tphi(1_Aa) = \vphi(1_A,a) = f(1_A)f(a) $$
for all $a \in A$, and by combining the above equations, we arrive at
$$f(a_1)f(a_2) = f(1_A)f(a_1a_2)$$
for all $a_1,a_2 \in A$.
As a corollary, if $A$ is zero product determined, and if $1_A \in A$ and $1_B \in B$ are identities, then any zero product preserving linear map $f: A \TO B$ that satisfies $f(1_A)=f(1_B)$ is an algebra homomorphism.

The initial work of Bre\v{s}ar, Gra\v{s}i\v{c}, and Ortega and subsequent work by Ge, Gra\v{s}i\v{c}, Li, and Wang  have provided examples both of zero product determined algebras and algebras that are not zero product determined \cite{Grasic10,Wang11b}. Recent similar work includes Bre\v{s}ar's and \v{S}merl's study of commutativity preserving linear map \cite{Bresar06} and Chen's, Wang's, and Yu's study of idempotent preserving bilinear maps and the notion of an idempotent elements determined algebra \cite{Wang11a}.

In the context of Lie algebras, Gra\v{s}i\v{c} in \cite{Grasic10} gives an example of an infinite-dimensional Lie algebra that is not zero product determined.
Bre\v{s}ar, Gra\v{s}i\v{c}, and Ortega produce a family of Lie algebras of arbitrarily-large finite dimension that are not zero product determined in \cite{Bresar09}.
In \cite{Grasic10}, Gra\v{s}i\v{c} shows that the classical Lie algebras over an arbitrary field are zero product determined.
Chen, Wang, and Yu compliment this result in \cite{Wang11b} by showing that every parabolic subalgebra of a finite-dimensional simple Lie algebra over an algebraically-closed field of characteristic zero is zero product determined.
In particular, the simple Lie algebras themselves are zero product determined, assuming the scalar field to be algebraically closed and of characteristic zero.

In this paper, we provide three main results that further research in this direction. We reformulate the definition of a zero product determined algebra in terms of tensor products and obtain necessary and sufficient conditions for an algebra to be zero product determined (Theorem \ref{zpd charactarization}).
These necessary and sufficient conditions allow us to prove that the direct sum of algebras $\bigoplus_{i \in I} A_i$ for any index set $I$ is zero product determined if and only if each component algebras $A_i$ is zero product determined (Theorem \ref{direct sums}).
We then apply this result in the context of Lie algebras, showing that all finite-dimensional reductive Lie algebras over an algebraically-closed field of characteristic zero, and all parabolic subalgebras of such Lie algebras, are zero product determined (Theorem \ref{reductive lie algebras are zpd}).

In all that follows, let $\K$ denote a fixed commutative ring.
By a \emph{module} we mean a $\K$-module. The  notation $S \leq
A$ means $S$ is a submodule of $A$. If $A$ is a module and $S
\subset A$, let $\langle S \rangle$ denote the submodule of $A$
generated by $S$.

By a \emph{linear map}, we mean a $\K$-linear map between two modules. By a
\emph{bilinear map}, we mean a $\K$-bilinear map between the Cartesian
product of two modules and a third module. We say $\vphi$ \emph{factors
through} $\psi$ to mean that $\vphi = \tphi \psi$ for some linear map $\tphi$.

\section{Preliminaries}

We will often make use of the following classical result in module theory.

\begin{lem}\label{factor through an epimorphism}
Let $A, B, C$ be modules.  Let $\vphi : A \TO B$ be a linear map, and let $\rho: A \TO C$
be a surjective linear map. There exists a linear map $\tphi:C\TO B$
satisfying $\vphi=\tphi\rho$ if and only if $\Krn \rho \subset \Krn
\vphi$, and in such a case, the map $\tphi$ is uniquely determined.
$$
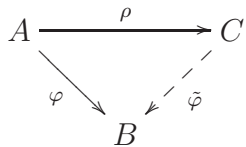
\xymatrix{
A \ar[rr]^\rho \ar[dr]_\vphi & & C \ar@{-->}[dl]^{\tphi} \\
& B & \\ }
$$
\end{lem}

\begin{proof}
For sufficiency, one is referred to the proof given of Theorem 7 in Chapter VI, Section 4 of \cite{Birkhoff65}.

For necessity, simply note that if such a $\tphi$ exists, and if $a \in \Krn \rho$, then
$$\vphi(a) = \tphi \rho (a) = \tphi (0) = 0.$$
\end{proof}

The \emph{tensor product} of two modules $A$ and $B$ is denoted by
$$
A \otimes B = A \otimes_{\K} B = \mathcal{F}(A \times B) / S
$$
where $\mathcal{F}(A \times B)$ is the free module taking for
generating set the Cartesian product $A \times B$, and $S$ is the
submodule generated by elements of the type
$(a_1+ka_2,b)-(a_1,b)-k(a_2,b)$ and $ (a,b_1+kb_2)-(a,b_1)-k(a,b_2)
$ for all $k \in \K,\ a, a_1, a_2 \in A,\  b, b_1, b_2 \in B$. We write $a
\otimes b$ to denote the coset $(a,b) + S$, and we call elements of
this type \emph{pure tensors}. Let
$$T=\{a\otimes b\mid a\in A,\ b \in B \}$$ denoted the set of all pure tensors. Every element of $A
\otimes B$ is a sum of finitely-many pure tensors.

Lemma \ref{factor through an epimorphism} implies the following
result.

\begin{lem}\label{tensor correspondence}
Let $A$, $B$, $C$ be modules. There is a one-to-one correspondence
between the set of bilinear maps $\vphi: A \times B \TO C$ and the
set of linear maps $\bphi : A \otimes B \TO C$, given by
$$\bphi\left(\sum_{i}a_i\otimes b_i\right)=\sum_{i}\vphi(a_i,b_i)\quad\text{and}\quad \vphi(a,b)=\bphi(a\otimes b)$$
for all $a, a_i\in A$ and $b, b_i\in B$.
\end{lem}

\begin{proof}
Suppose $\bphi : A \otimes B \TO C$ is a linear map. Let
$P(\bphi):A\times B\TO C$ be a map defined by
$P(\bphi)(a,b)=\bphi(a\otimes b)$ for all $a\in A$ and $b\in
B$. Obviously $P(\bphi)$ is a well-defined bilinear map.

Conversely, suppose $\vphi: A \times B \TO C$ is a bilinear map.
Since $\mathcal{F}(A \times B)$ is the free module on $A \times B$,
the map $\vphi$  induces a linear map
$\mathcal{F}(\vphi):\mathcal{F}(A \times B)\TO C$, such that
$$\mathcal{F}(\vphi)\left( \sum_i k_i(a_i,b_i) \right) = \sum_i k_i \vphi(a_i,b_i),
\qquad  k_i\in \K.$$

The canonical projection $\pi:\mathcal{F}(A\times B)\TO
\mathcal{F}(A\times B)/S=A\otimes B$ is a linear map with
$\Krn\pi=S$. We claim that $S\subset\Krn\mathcal{F}(\vphi)$.
$(a_1+ka_2,b)-(a_1,b)-k(a_2,b)$ is a typical generator of $S$, and
\begin{align*}
&\mathcal{F}(\vphi)\left( (a_1+ka_2,b)-(a_1,b)-k(a_2,b) \right) \\
&= \vphi(a_1+ka_2,b) -\vphi(a_1,b) -k\vphi(a_2,b) = 0.
\end{align*}
The derivation is similar for the other typical generator
$(a,b_1+kb_2)-(a,b_1)-k(a,b_2)$, verifying the claim.

By Lemma \ref{factor through an epimorphism},  there exists a linear
map $Q(\vphi):A\otimes B\TO C$ such that
$\mathcal{F}(\vphi)=Q(\vphi)\pi$. We see that
\begin{eqnarray*}
 Q(\vphi)\left(\sum_{i}a_i\otimes b_i\right)
&=& Q(\vphi)\pi\left(\sum_{i}(a_i, b_i)\right)
\\
&=& \mathcal{F}(\vphi)\left(\sum_{i}(a_i,
b_i)\right)=\sum_{i}\vphi(a_i, b_i)
\end{eqnarray*}
for all $a_i\in A$ and $b_i\in B$.

To prove the one-to-one correspondence  in the statement, it remains
to show that the functors $\bphi\To P(\bphi)$ and $\vphi\To
Q(\vphi)$ are inverse to each other. On one hand, given the linear
map $\bphi : A \otimes B \TO C$, we have
\begin{eqnarray*}
Q\left(P(\bphi)\right)\left(\sum_{i}a_i\otimes b_i\right) &=&
\sum_{i}P(\bphi)(a_i, b_i) \\
&=& \sum_{i}\bphi(a_i\otimes b_i)=\bphi\left(\sum_{i}a_i\otimes
b_i\right)
\end{eqnarray*}
and thus $Q\left(P(\bphi)\right)=\bphi$; on the other hand,
given the bilinear map  $\vphi:A\times B\TO C$, we have
\begin{eqnarray*}
P\left(Q(\vphi)\right)(a,b)=Q(\vphi)(a\otimes b)=\vphi(a,b)
\end{eqnarray*}
and thus $P\left(Q(\vphi)\right)=\vphi$. This completes the proof.
\end{proof}

\section{Zero product determined algebras}

\begin{defn}
An \emph{algebra} (more precisely, a $\K$-algebra) is a pair $(A,\mu)$ where $A$ is a module and $\mu: \A \TO A$ is a linear map.
\end{defn}

This definition encompasses associative algebras, alternative algebras (an example being the octonions), Leibniz algebras, Lie algebras, and Jordan algebras, among others.
The algebra multiplication is encoded by the map $\mu$ if we define   $a_1a_2 = \mu(a_1 \otimes a_2)$.
In light of Lemma \ref{tensor correspondence}, the condition $\mu: \A \TO A$ is equivalent to the requirement that multiplication be bilinear.
We will denote by $A^2$ the submodule $\Img \mu \leq A$.
Notice that $A^2$ consists of finite sums of products in $A$, not merely the products themselves.

\begin{defn}
The algebra $(A,\mu)$ is called \emph{zero product determined} if whenever a linear map $\vphi: \A \TO B$ satisfies
$$\mu(a_1 \otimes a_2) = 0 \ \textrm{ implies } \ \vphi(a_1 \otimes a_2) = 0,$$
then $\vphi$ factors through $\mu$:
$$
\xymatrix{
\A \ar[rr]^\mu \ar[dr]_\vphi & & A^2 \ar@{-->}[dl]^{\tphi} \\
& B & \\ }
$$
\end{defn}

By Lemma \ref{tensor correspondence}, this definition is in agreement with that given in \cite{Bresar09}.

We will reformulate this definition in terms of kernels and pure tensors below. First, we require some terminology.

Consider an algebra $(A, \mu)$, and recall that   $T = \set{a_1 \otimes a_2}{a_1, a_2 \in A} $ denotes the set of pure tensors.
For any map $\psi: \A \TO B$, let $$T_\psi = T \cap \Krn \psi.$$
In other words, $T_\psi$ is the set comprised of those pure tensors whose image under $\psi$ is 0. In particular, $T_\mu$ is the set of pure tensors of members of $A$ whose product in $A$ is 0.

\begin{lem}\label{zpd reformulation}
$(A, \mu)$ is zero product determined if and only if for every linear map $\vphi: \A \TO B$, we have
$$
T_\mu \subset T_\vphi \ \textrm{ implies that } \ \Krn \mu \subset \Krn \vphi.
$$
\end{lem}

\begin{proof}
The two conditions "$T_\mu \subset T_\vphi$" and "$\mu(a_1 \otimes a_2) = 0 \text{ implies } \vphi(a_1 \otimes a_2) = 0$" are literally identical, and "$\Krn \mu \subset \Krn \vphi$" is equivalent to $\vphi$ factoring through $\mu$ by Lemma \ref{factor through an epimorphism}.
\end{proof}

Our first main result, stated below, gives a very clear conception of what makes an algebra $(A, \mu)$ zero product determined. First, recall that $\langle T_{\mu} \rangle$ denotes the module generated by $T_{\mu}$.

\begin{thm}\label{zpd charactarization}
$(A, \mu)$ is zero product determined if and only if
$$
\langle T_{\mu} \rangle = \Krn \mu.
$$
\end{thm}

\begin{proof}
If $\langle T_{\mu} \rangle = \Krn \mu$, then $T_\mu \subset T_\vphi$ implies that $\Krn \mu = \langle T_\mu \rangle \leq \langle T_\vphi \rangle \leq \Krn \vphi$, so $\vphi$ factors through $\mu$ whenever $T_\mu \subset T_\vphi$.
On the other hand, if $\langle T_{\mu} \rangle \neq \Krn \mu$, the canonical projection $\pi : \A \TO (\A) / \langle T_\mu \rangle$ satisfies $T_\mu \subset T_\pi$ but by Lemma \ref{factor through an epimorphism}, $\pi$ does not factor through $\mu$, so $(A, \mu)$ is not zero product determined.
\end{proof}

The condition $\langle T_{\mu} \rangle = \Krn \mu$ is just to say that $\Krn \mu$ is generated by its intersection with $T$, i.e. by pure tensors.
Thus, the theorem says that an algebra $A$ is zero product determined if and only if the kernel of its multiplication map is generated by pure tensors.
It is worthwhile to observe that while $\A$ is generated by $T$, an arbitrary submodule $B \leq \A$ is not necessarily generated by its intersection with $T$.
We give a few examples to illustrate this point.

\begin{prop}\label{tensor and symmetric algebras are not zpd}
Suppose $\K$ is a field. Let $V$ be a vector space over $\K$ with $\dim V \geq 2$.
Then, neither the tensor algebra $\mathcal T (V)$ nor the symmetric algebra $\mathcal S (V)$ are zero product determined.
\end{prop}

\begin{proof}
Recall that the tensor algebra $\mathcal T (V)$ over a vector space $V$ may be thought of as the free associative $\K$-algebra on $\dim V$ generators. Likewise, the symmetric algebra $\mathcal S (V)$ over $V$ may be thought of as the free commutative associative $\K$-algebra on $\dim V$ generators.

The tensor and symmetric algebras over a vector space are integral domains, and as such, we have that $\mu(t_1 \otimes t_2) = 0$ if and only if $t_1 = 0$ or $t_2 = 0$, in either case giving $t_1 \otimes t_2 = 0$. In short, this means $T_\mu = 0$. To show that these algebras are not zero product determined, we must now show that $\Krn \mu \neq 0$. In other words, we must produce at least one non-trivial kernel element.

Since $\dim V \geq 2$, we may select two $\K$-linearly independent vectors $v_1,v_2 \in V$.

 Then for the tensor algebra $\mathcal T (V)$, consider
the element
$$v_1v_2 \otimes v_1 - v_1 \otimes v_2v_1 \in \mathcal T (V) \otimes \mathcal T (V).$$
(Here, multiplication in $\mathcal T (V)$ is denoted by juxtaposition, to avoid confusion with members of $\mathcal T (V) \otimes \mathcal T (V)$.)
We have that
$$\mu(v_1v_2 \otimes v_1 - v_1 \otimes v_2v_1) = v_1v_2v_1 - v_1v_2v_1 = 0$$
so that $v_1v_2 \otimes v_1 - v_1 \otimes v_2v_1 \in \Krn \mu$, while also $v_1v_2 \otimes v_1 - v_1 \otimes v_2v_1 \neq 0$ by linear independence.

As for the symmetric algebra $\mathcal S (V)$, we may use the element
$$v_1 \otimes v_2 - v_2 \otimes v_1 \in \mathcal S (V) \otimes \mathcal S (V)$$
which is non-zero yet contained in $\Krn \mu$ by commutativity of $\mathcal S (V)$.
\end{proof}

The next proposition shows that the condition on the dimension of $V$ in the previous is quite necessary. We recall that $\K$ is an arbitrary commutative ring. When viewed as a module, $\K$ has an obvious algebra structure by letting $\mu(k_1 \otimes k_2)=k_1k_2$. Furthermore, any ideal $H$ of $\K$ is a submodule of $\K$, and the quotient module $\K /H$, again through the usual multiplication, has an algebra structure.

\begin{prop}\label{cyclic modules are zpd}
Assume $1 \in \K$. For any ideal $H \leq \K$, the ring $\K / H$ is zero product determined when viewed as a $\K$-algebra. In particular, $\K$ is zero product determined as a $\K$-algebra.
\end{prop}

\begin{proof}
$\K /H \otimes_\K \K /H \cong \K /H$, with $\mu$ serving as an isomorphism. This is because
\begin{align*}
(k_1 + H) \otimes_\K (k_2 + H) &= k_1(1+H) \otimes_\K (k_2 + H) \\
&= (1 + H) \otimes_\K k_1(k_2+ H) \\
&= (1 + H) \otimes_\K (k_1k_2 + H), \\
\end{align*}
so $k_1k_2+H = 0 + H$ (which is to say $\mu ((k_1+H) \otimes_\K (k_2+H)) = 0$) implies that
\begin{align*}
(k_1+H) \otimes_\K (k_2+H) &= (1+H) \otimes_\K (k_1k_2 +H) \\
&= (1+H) \otimes_\K (0+H) \\
&= 0. \\
\end{align*}
In particular, we have that $\Krn \mu = 0$, so $\K /H$ is trivially zero product determined.
\end{proof}

\section{Direct Sums of Algebras}

In order to state and prove our second main result, we require the following terminology.
Given algebras $(A, \mu)$ and $(B, \lambda)$, we may endow their module direct sum $A \oplus B$ with an algebra structure.
Define
$$ \mu \boxplus \lambda : (A \oplus B) \otimes (A \oplus B) \TO A \oplus B $$
by
$$ \mu \boxplus \lambda ((a_1,b_1) \otimes (a_2,b_2)) = (\mu(a_1 \otimes a_2),\lambda(b_1 \otimes b_2)).$$
$\mu \boxplus \lambda$ is seen to be well-defined by Lemma \ref{tensor correspondence} after noting that the function
$$\overline{\mu \boxplus \lambda}((a_1,b_1),(a_2,b_2)) = (\mu(a_1 \otimes a_2),\lambda(b_1 \otimes b_2))$$
is bilinear.
In this way, $(A \oplus B, \mu \boxplus \lambda)$ is an algebra.
This agrees with the usual meaning of the direct sum of two algebras using component-wise multiplication, since
$$(a_1,b_1)(a_2,b_2) = (a_1a_2,b_1b_2) = \mu \boxplus \lambda ((a_1,b_1) \otimes (a_2,b_2)).$$

The above example illustrates how the direct sum of algebras can be constructed completely in terms of linear maps on tensor products. Our primary purpose in this section is to show that the direct sum $\left(\bigoplus_{i\in I}A_i,\boxplus_{i\in I}\mu_i\right)$ of an arbitrary set of algebras is zero product determined if and only if each component algebra $(A_i, \mu_i)$ is zero product determined.

\begin{lem}\label{canonical isomorphism}
Let $A_i$ ($i \in I$) be modules.
$$ \left( \bigoplus_{i \in I} A_i \right) \otimes \left( \bigoplus_{i \in I} A_i \right) \cong \bigoplus_{i,j \in I} A_i \otimes A_j $$
\end{lem}

\begin{proof}
The lemma is proved after making two application of Proposition 2.1 in Chapter XVI of \cite{Lang02}.
\end{proof}

We will rely on this isomorphism freely without mention in what follows. In particular, we will use the lemma to define $\boxplus_{i\in I}\mu_i$ and derive some simple properties.

\begin{defn}
Given algebras $(A_i,\mu_i)$ $(i\in I)$, their \emph{direct sum} as an algebra is the pair $(A,\mu)$ where $A = \bigoplus_{i\in I}A_i$, and
$$\mu = \boxplus_{i\in I}\mu_i : \bigoplus_{i,j \in I} A_i \otimes A_j \TO A$$
is defined by
$$ \mu|_{A_i \otimes A_j} = \left\{
     \begin{array}{lr}
       \mu_i & : i = j \\
       0 & : i \neq j \\
     \end{array}
   \right.$$
\end{defn}

Since $\mu|_{A_i \otimes A_j}$ is a linear map on each component $A_i \otimes A_j$ of $A$, the map $\mu = \bigoplus_{i,j} \mu_{A_i \otimes A_j}$ is a well-defined linear map on $A$ by the universal property of module direct sums. In particular, we note the fact that
$$ \mu(A_i \otimes A_i) \subset A_i.$$

Our second main result below shows that a direct sum of algebras is zero product determined if and only if each of its components is.

\begin{thm}\label{direct sums}
  Let $(A_i,\mu_i)$ $(i\in I)$ be algebras. Let $(A, \mu)$ be their direct sum as defined above. Then $(A, \mu)$ is zero product determined if and only if $(A_i, \mu_i)$ is zero product determined for all $i \in I$.
\end{thm}

To prove Theorem \ref{direct sums}, we need the following two lemmas about $(A, \mu)$.

\begin{lem}\label{direct sums kernel}
  $\displaystyle \Krn \mu=\left(\bigoplus_{i\in I}\Krn \mu_i\right)\oplus\left(\bigoplus_{i,j\in I,\ i\ne j}A_i\otimes A_j\right)$.
\end{lem}

\begin{proof}
    Clearly $ \displaystyle \Krn \mu \supset \left(\bigoplus_{i\in I}\Krn \mu_i\right)\oplus\left(\bigoplus_{i,j\in I,\ i\ne j}A_i\otimes A_j\right) $. We will demonstrate the reverse inclusion.
    Denote by $p_{i,j}$ the canonical projection $p_{i,j} : A \TO A_i \otimes A_j$. Let $a \in \Krn \mu$. Then
    \begin{align*}
        0 &= \mu (a) \\
        &= \sum_{i,j \in I} \mu \left( p_{i,j}(a) \right) \\
        &= \sum_{i \in I} \mu(p_{i,i}(a)). \\
    \end{align*}
    Each $\mu(p_{i,i}(a)) \in A_i$, so each $\mu(p_{i,i}(a)) = 0$ by linear independence, giving $ \displaystyle a \in \left(\bigoplus_{i\in I}\Krn \mu_i\right)\oplus\left(\bigoplus_{i,j\in I,\ i\ne j}A_i\otimes A_j\right)$.
\end{proof}

\begin{lem}\label{direct sums pure tensor}
  $\displaystyle \langle T_{\mu}\rangle
      = \left(\bigoplus_{i\in I}\langle T_{\mu_i}\rangle\right)\oplus\left(\bigoplus_{i,j\in I,\ i\ne j}A_i\otimes A_j\right)$.
\end{lem}

\begin{proof}
By Lemma \ref{direct sums kernel}, $ T_{\mu_i} \subset T_{\mu}$ for $i\in I$ and $A_i\otimes A_j\subset \langle T_{\mu}\rangle$ for $i,j\in I$ and $i\ne j$. Therefore,
      $$\left(\bigoplus_{i\in I}\langle T_{\mu_i}\rangle\right)\oplus\left(\bigoplus_{i,j\in I,\ i\ne j}A_i\otimes A_j\right)\subset \langle T_{\mu}\rangle.
      $$

      Conversely, every element of $T_{\mu}$ is of the form
      $$a\otimes a'=\left(\sum_{i\in I}a_i\right)\otimes\left(\sum_{i\in I}a_i'\right)=\sum_{i, j\in I}a_i\otimes a_j',$$
      where $a_i,a_i'\in A_i$, $a_i=0$ and $a_i'=0$ for all but finitely many $i$, and
      $$0=\mu(a\otimes a')=  \sum_{i,j\in I}\delta_{ij}\mu_i(a_i\otimes a_j')
      =\sum_{i\in I}\mu_i(a_i\otimes a_i').$$
      Therefore, $a_i\otimes a_i'\in T_{\mu_i}$ and
       $$a\otimes a'\in
       \left(\bigoplus_{i\in I}\langle T_{\mu_i}\rangle\right)\oplus\left(\bigoplus_{i,j\in I,\ i\ne j}A_i\otimes A_j\right)
       .$$
      This completes the proof.
\end{proof}

\begin{proof}[Proof of Theorem \ref{direct sums}]
By Theorem \ref{zpd charactarization}, $(A,\mu)$ is zero product determined if and only if $\Krn \mu=\langle T_{\mu}\rangle$.
By Lemmas \ref{direct sums kernel} and \ref{direct sums pure tensor},
$\Krn \mu=\langle T_{\mu}\rangle$ if and only if
$$
\left(\bigoplus_{i\in I}\Krn \mu_i\right)\oplus\left(\bigoplus_{i,j\in I,\ i\ne j}A_i\otimes A_j\right)
=
\left(\bigoplus_{i\in I}\langle T_{\mu_i}\rangle\right)\oplus\left(\bigoplus_{i,j\in I,\ i\ne j}A_i\otimes A_j\right),
$$
that is, if and only if
$\Krn \mu_i=\langle T_{\mu_i}\rangle$
for all $i\in I$,
and in turn
$\Krn \mu_i=\langle T_{\mu_i}\rangle$
for all $i\in I$ if and only if
$(A_i,\mu_i)$ is zero product determined for all $i\in I$.
\end{proof}

\section{Applications to Lie algebras}

An abelian Lie algebra $A$ is a Lie algebra with trivial multiplication. In other words, $\Krn \mu = \A$, so that $A$ is trivially seen to be zero product determined.

\begin{lem}\label{abelian Lie algebras are zpd}
Let $A$ be an abelian Lie algebra. Then $A$ is zero product determined.
\end{lem}

\begin{proof}
That $A$ is abelian means that $T_\mu = T$, so that
$$\langle T_\mu \rangle = \langle T \rangle = \A = \Krn \mu.$$
By Theorem \ref{zpd charactarization}, $A$ is zero product determined.
\end{proof}

Recall that a semi-simple Lie algebra $L$ decomposes as the direct sum of simple ideals $L = L_1 \oplus \cdots \oplus L_r$ and that a reductive Lie algebra $L$ decomposes as $L = L_0\oplus L_1\oplus\cdots\oplus L_r$, where $L_0$ is the center of L and $L_1\oplus\cdots\oplus L_r$ is semi-simple. See \cite{Bourbaki75} and \cite{Humphreys72} for details.

In \cite{Wang11a}, Wang, et. al. show  that every parabolic subalgebra of a finite-dimensional simple Lie algebra over an algebraically-closed field of characteristic zero is zero product determined.
Combining this result with our result for direct sums (Theorem \ref{direct sums}) and the above lemma significantly broadens the class of Lie algebras known to be zero-product determined.  Below is our third main result.

\begin{thm}\label{reductive lie algebras are zpd}
Let $L$ be a reductive Lie algebra over an algebraically closed field $\K$ of characteristic 0. Then every parabolic subalgebra $P$ of $L$ is zero product determined. In particular, $L$   is zero product determined.
\end{thm}

\begin{proof}
The reductive Lie algebra $L$ can be decomposed as $L=L_0\oplus L_1\oplus\cdots\oplus L_r$, where $L_0=Z(L)$ is an abelian ideal, and $L_1,\cdots, L_r$ are simple ideals of $L$.
Every parabolic subalgebra $P$ of $L$ is isomorphic by some automorphism to a standard parabolic subalgebra $P'$ of the form $L_0\oplus P_1\oplus\cdots\oplus P_r$, where $P_i$ is a parabolic subalgebra of $L_i$ for $i=1,\cdots, r$. Then $P'$ is zero product determined by Theorem \ref{direct sums}, whence so is $P$.
\end{proof}



\begin{bibdiv}

\begin{biblist}

\bib{Alaminos09}{article}{
    title={Maps preserving zero products},
    author={Alaminos, J.},
    author={Bre\v{s}ar, M.},
    author={Extremera, J.},
    author={Villena, A. R.},
    journal={Studia Mathematica},
    volume={193},
    date={2009},
    pages={131--159}}

\bib{Atiyah69}{book}{
    title={Introduction to Commutative Algebra},
    author={Atiyah, Michael F.},
    author={Macdonald, Ian G.},
    date={1969},
    publisher={Addison-Wesley},
    address={Reading Massachusetts}}

\bib{Birkhoff65}{book}{
    title={Algebra},
    author={Birkhoff, Garrett}
    author={Mac Lane, Saunders},
    date={1965},
    publisher={Macmillan},
    address={New York}}

\bib{Bourbaki75}{book}{
    title={Elements of mathematics. Lie groups and Lie algebras},
    author={Bourbaki, N.},
    date={1975},
    publisher={Addison-Wesley}}

\bib{Bresar09}{article}{
    title={Zero product determined matrix algebras},
    author={Bre\v{s}ar, Matej},
    author={Gra\v{s}i\v{c}, Mateja},
    author={Ortega, Juana S\'{a}nchez},
    journal={Linear Algebra and its Applications},
    volume={430},
    date={2009},
    pages={1486--1498}}

\bib{Bresar06}{article}{
    title={On bilinear maps on matrices with applications to commutativity preservers},
    author={Bre\v{s}ar, Matej},
    author={\v{S}emrl, Peter},
    journal={Journal of Algebra},
    volume={301},
    date={2006},
    pages={803--837}}

\bib{Chebotar03}{article}{
    title={Mappings preserving zero products},
    author={Chebotar, M. A.},
    author={Ke, W.-F.},
    author={Lee, P.-H.},
    author={Wong, N.-C.},
    journal={Studia Mathematica},
    volume={155},
    date={2003},
    pages={77--94}}

\bib{Chebotar06}{article}{
    title={On maps preserving zero Jordan products},
    author={Chebotar, Mikhail A.},
    author={Ke, Wen-Fong},
    author={Lee, Pjek-Hwee},
    author={Zhang, Ruibin},
    journal={Monatsheft f\"{u}r Mathematik},
    volume={149},
    date={2009},
    pages={91-101}}

\bib{Grasic10}{article}{
    title={Zero product determined classical Lie algebras},
    author={Gra\v{s}i\v{c}, Mateja},
    journal={Linear and Multilinear Algebra},
    volume={58},
    date={2010},
    pages={1007--1022}}

\bib{Humphreys72}{book}{
    title={Introduction to Lie Algebras and Representation Theory},
    author={Humphreys, James E.},
    date={1974},
    publisher={Springer},
    address={New York}}

\bib{Lang02}{book}{
    title={Algebra},
    author={Lang, Serge},
    date={2002},
    publisher={Springer},
    address={New York}}

\bib{Wang11a}{article}{
    title={Idempotent elements determined matrix algebras},
    author={Wang,Dengyin},
    author={Li,Xiaowei},
    author={Ge,Hui},
    journal={Linear Algebra and its Applications},
    volume={},
    date={preprint},
    pages={}}

\bib{Wang11b}{article}{
    title={A class of zero product determined Lie algebras},
    author={Wang, Dengyin},
    author={Yu, Xiaoxiang},
    author={Chen, Zhengxin},
    journal={Journal of Algebra},
    volume={331},
    date={2011},
    pages={145--151}}

\end{biblist}

\end{bibdiv}
\end{document}